\documentclass{amsart}

\usepackage{amssymb, amsmath, amsthm, hyperref, euscript}

\usepackage{amscd}

\usepackage{graphicx}

\usepackage[english]{babel}

\author{Rafael Torres}

\title[Smooth structures, Gluck twists, and involutions]{Smooth structures on non-orientable four-manifolds and free involutions}

\address{Mathematical Institute - University of Oxford, Andrew Wiles Building\\Radcliffe Observatory Quarter, Woodstock Road\\Oxford\\OX2 6GG\\England}

\address{Scuola Internazionale Superiori di Studi Avanzati\\ Via Bonomea 265\\34136\\Trieste\\Italy}

\email{rtorres@sissa.it}



\theoremstyle{plain}

\newtheorem{theorem}{Theorem}

\newtheorem{corollary}{Corollary}

\newtheorem{proposition}{Proposition}

\newtheorem{lemma}{Lemma}

\newtheorem{remark}[equation]{Remark}

\newtheorem{convention}{Convention}

\theoremstyle{definition}

\newtheorem{definition}{Definition}

\newcommand{\Q}{\mathbb{Q}}

\newcommand{\R}{\mathbb{R}}

\newcommand{\Z}{\mathbb{Z}}

\newcommand{\N}{\mathbb{N}}

\newcommand{\C}{\mathbb{C}}

\begin{document}

\maketitle

\emph{Abstract}: In this paper, we investigate existence of inequivalent smooth structures on closed smooth non-orientable 4-manifolds building upon results of Akbulut, Cappell-Shaneson, Fintushel-Stern, Gompf, and Stolz. We add to the number of known constructions and provide new examples of exotic manifolds that are obtained as an application of Gluck twists to the standard smooth structure. Inspection of the smooth structure on the oriented 2-covers yields existence results of orientation-reversing exotic free involutions.


\section{Introduction and main results}

Two smooth manifolds $X$ and $Y$ that represent two different diffeomorphism classes within a same homeomorphism class are called inequivalent smooth structures. Fixing the homeomorphism type of $X$, the manifold $Y$ is called an exotic copy of $X$ and/or an exotic smooth structure on $X$, and the smooth structure on $X$ is said to be standard. The first example of such a smooth structure in dimension four was constructed on the topological type of the real projective 4-space in \cite{[CS]}, where a collection of manifolds that are simply homotopy equivalent but not smoothly s-cobordant to $\R P^4$ was constructed (cf. Section \ref{Section FRP}). Another inequivalent smooth structure on $\R P^4$ was constructed in \cite{[FS]}. The smooth structures remain inequivalent under connected sums with an arbitrary number of copies of $S^2\times S^2$ \cite{[CS], [FS]}, yet they become diffeomorphic after forming the connected sum with a single copy of the complex projective space $\mathbb{CP}^2$ \cite{[AO]}.\\

The papers \cite{[CS], [FS]} use topological invariants to distinguish the smooth structures. Following a suggestion in \cite{[Gi]}, a spectral invariant known as the $\eta$-invariant was used in \cite{[S]} to descry the smooth structure built in \cite{[CS]} from the standard one in the homeomorphism class of $\R P^4$. The $\eta$-invariant mod $\Z$ is a homeomorphism invariant \cite{[HKT]} and its value mod $2\Z$ completely determines $Pin^+$-bordism classes in $\Omega^{Pin^+}_4$ \cite{[S]} . The $\eta$-invariant of the exotic $\R P^4$ that was constructed in \cite{[FS]} was computed in \cite{[O]}.\\

We observe that coupling results in \cite{[CS], [AO], [S]} yields the following theorem. We denote the connected sum of two manifolds $M_1$ and $M_2$ by $M_1\# M_2$, the circle sum along the loop that represents the generator of the fundamental group by $M_1\#_{S^1} M_2$ (see Definition \ref{Definition CS}) , and the real Hopf bundle over $\R P^2$ by $\gamma$.

\begin{theorem}{\label{Theorem 0}} There are exactly two inequivalent smooth structures on the non-orientable 4-manifolds
\begin{center}
$S(2\gamma \oplus \R)\#(k - 1)(S^2\times S^2)$ and $\#_{S^1} r\cdot (\R P^4) \# (k - 1)(S^2\times S^2)$
\end{center}
for $1\leq r \leq 3$ and any $k\in \N$ that are discerned by the $\eta$-invariant.

The corresponding manifolds become diffeomorphic after taking a connected sum with $\mathbb{CP}^2$.
\end{theorem}

Handlebodies of the smooth non-orientable 4-manifolds of Theorem \ref{Theorem 0} are constructed in Section \ref{Section Hand} following the analysis done in \cite{[AO], [AL]}. The respective calculations and limitations of the $\eta$-invariant are studied in Section \ref{Section eta1} and Section \ref{Section eta2}.\\

The inequivalent smooth structure on $\R P^4$ of \cite{[CS]} was used to build an exotic copy of $S^3\widetilde{\times} S^1 \# S^2\times S^2$ in \cite{[AO]} (cf. \cite{[FS1]}), where $S^3\widetilde{\times} S^1$ denotes the non-orientable 3-sphere bundle over the circle. Denote this exotic smooth structure by $A$. Studying the handlebody of $A$ \cite[Figure 7]{[AT]}, it was shown in \cite{[AT]} that this exotic manifold is obtained by applying a Gluck twist \cite{[Gl]} along an embedded 2-sphere in $S^3\widetilde{\times} S^1 \# S^2\times S^2$. This was the only known example where Gluck twists do produce an exotic manifold (cf. \cite{[G0]}).\\

We compute the $\eta$-invariant of both smooth structures on the homeomorphism class of $S^3\widetilde{\times} S^1 \# S^2\times S^2$, and then use the exotic manifold $A$ to build more exotic manifolds as circle sums; the smooth structures that we construct are distinguished from the standard ones by the $\eta$-invariant. The handlebody of $A$ \cite{[AT]} allows us to conclude that the exotic copies that are obtained in this way can be described as a result of performing Gluck twists to the standard smooth structure. We thus obtain the following result.

\begin{theorem}{\label{Theorem Pr}} Let $X$ be a closed smooth non-orientable 4-manifold with a $Pin^+$-structure $\phi_X$. There exists a manifold $Y$ homeomorphic to $X\# S^2\times S^2$ with a $Pin^+$-structure $\phi_Y$ such that
\begin{center}
$\eta(Y, g_Y, \phi_Y) = \eta(X, g_X, \phi_X) \pm 1$ mod $2\Z$
\end{center}

for all Riemannian metrics $g_Y$ and $g_X$ on $Y$ and $X$ respectively.

Moreover, there is an embedding $f: S^2\hookrightarrow X\# S^2\times S^2$ such that the manifold $Y$ is obtained by performing a Gluck twist to $X\# S^2\times S^2$ along $f(S^2)$.
\end{theorem}

Theorem \ref{Theorem Pr} by itself is not enough to conclude the existence of an inequivalent smooth structure, since in such a scenario all values of the $\eta$-invariant are required to be different for all $Pin^+$-structures to distinguish the smooth structures. Examples of exotic 4-manifolds obtained by the procedure of Theorem \ref{Theorem Pr} are provided in our next theorem, where we denote the Klein bottle by $Kb$, and the non-orientable total space of the non-trivial  $S^2$-bundle over $Kb$ with vanishing second Stiefel-Whitney class by $\xi_3$.

\begin{theorem}{\label{Theorem M}} Suppose $k\in \N$, and let $X_i$ be a smooth 4-manifold of Theorem \ref{Theorem 0}. The manifolds\smallskip
\begin{center}\begin{enumerate}

\item $S^3\widetilde{\times} S^1 \# k(S^2\times S^2)$\smallskip\smallskip


\item $\xi_5 \# k(S^2\times S^2) = Kb\times S^2\# k(S^2\times S^2)$\smallskip\smallskip

\item $\xi_3 \# k(S^2\times S^2)$ \smallskip\smallskip

\item $Kb\times T^2$ $\# k(S^2\times S^2)$ \smallskip\smallskip

\item  $X_1\# X_2$

\end{enumerate}\end{center}
admit an inequivalent smooth structure that is detected by the $\eta$-invariant.

The exotic smooth structure of Items (1) - (4) is obtained by performing a Gluck twist along an embedded 2-sphere to the standard smooth structure. 

\end{theorem}

Theorem \ref{Theorem Pr} has the following advantages over the known methods to construct exotic non-orientable 4-manifolds \cite{[CS], [FS1], [K], [AO]}. Unlike the cut-and-paste procedure of \cite{[CS]} (See Section \ref{Section FRP}), it does not require the existence of an element of order two in the fundamental group of the manifold for it to be effective, albeit increasing the Euler characteristic by two. The construction in \cite{[K]} increases the Euler characteristic by twenty-two. Moreover, Theorem \ref{Theorem Pr} also yields more examples where Gluck twisting a smooth manifold does produce inequivalent smooth structures \cite{[AO], [AL]}. Theorem \ref{Theorem 0} is proven in Section \ref{Section Exotica}.

\begin{corollary}{\label{Corollary M}} \cite{[AO], [W]}. Every closed smooth non-orientable 4-manifold with infinite cyclic fundamental group, Euler characteristic at least six, and a $Pin^+$-structure admits an inequivalent smooth structure that is detected by the $\eta$-invariant.
\end{corollary}

It is natural to study the smooth structures on the orientation 2-covers of the exotic manifolds that are constructed. The universal cover of an exotic $\R P^4$ was shown to be standard in \cite{[FS]}. This was the first known example of a pair of orientation-reversing free involutions whose orbit spaces are homeomorphic, yet not diffeomorphic. It was shown in \cite{[G0]} that the universal cover of an exotic real projective 4-space built in \cite{[CS]} is standard, hence providing another such example (see \cite{[A], [G]} as well). Since the inequivalent smooth structures on $\R P^4$ are invariant under connected sums with any number of copies of $S^2\times S^2$, these results yield examples of such involutions on the connected sums of an even number of $S^2\times S^2$.\\

Our next theorem extends these results by showing that the smooth structure is the standard one on the orientation 2-covers of the manifolds of Theorem \ref{Theorem 0} and Theorem \ref{Theorem M} by building upon \cite{[G0], [FS]}. We will use the following definition to state it. We are indebted to Danny Ruberman for its motivation.

\begin{definition}{\label{Definition Inv}} Let $A_G$ and $B_G$ be closed smooth n-manifolds with fundamental group $\pi_1(A_G) = G = \pi_1(B_G)$, let $\varphi_{A_G}$ and $\varphi_{B_G}$ be involutions
\begin{center}
$C\overset{\varphi_{A_G}}\longrightarrow A_G$ and $C\overset{\varphi_{B_G}}\longrightarrow B_G$,

\end{center}

where $C$ is the orientation 2-cover. The involution \emph{$\varphi_A$ is an exotic copy of $\varphi_B$} if there exists a $G$-equivariant homeomorphism $A_g\rightarrow B_G$, but no such diffeomorphism. When such an exotic involution exists, we will say that $C$ admits an \emph{exotic/inequivalent $G$-involution} or an exotic/inequivalent involution by $G$.

\end{definition}

Through out the manuscript, the smooth manifolds $A_G$ and $B_G$ are non-orientable and four-dimensional, and the involutions $\varphi_{A_G}$ and $\varphi_{B_G}$ are free and orientation-reversing. Under Definition \ref{Definition Inv}, hence,\cite{[FS], [CS], [G]} constructed free orientation-reversing exotic $\Z/2$-involutions on $(2k - 2)(S^2\times S^2)\#S^4$ for every $k\in \N$.  A free orientation-reversing exotic $\Z$-involution on $S^3\times S^1\#S^2\times S^2\# S^2\times S^2$ was constructed in \cite{[FS1]}. We obtained the following extension in this paper. 

\begin{theorem}{\label{Theorem Inv}} Let $k\in \N$. The connected sums \smallskip
\begin{enumerate}
\item $(k - 1)(S^2\times S^2)\# S^4$\smallskip\smallskip

\item $S^3\times S^1 \# 2k(S^2\times S^2)$,\smallskip\smallskip

\item $T^2\times S^2 \# 2k(S^2\times S^2)$, \smallskip\smallskip

\item $T^2\times T^2 \# 2k(S^2\times S^2)$, and \smallskip\smallskip

\item $S^3\times S^1 \# 2(k - 1)(S^2\times S^2)$\smallskip
\end{enumerate}

admit free orientation-reversing exotic $G$-involutions by the groups $\Z/2$, $\Z$, $\Z \rtimes \Z$, $\Z\times \Z \times \Z \rtimes \Z$, and $\Z/2\ast \Z/2$, respectively.

\end{theorem}

Item (2) was proven in \cite{[FS1], [AO], [G0]}. Item (3) contains examples of exotic involutions for two different homeomorphism classes of orbit spaces (please see Remark \ref{Remark Orb}). The proof of Theorem \ref{Theorem Inv} is given in Section \ref{Section Inv}.

\subsection{Acknowledgements:} We thank Will Cavendish, Bob Gompf, Danny Ruberman, Ron Stern, Damiano Testa, and Yang Su for useful conversations or e-mail exchanges. We are indebted to an anonymous referee and to Ian Hambleton for kindly pointing out a mistake on a previous version of the manuscript. The Simons Foundation is gratefully acknowledged for its support, and the Erwin Schr\"odinger International Institute in Mathematical Physics for its hospitality. The illustrations were done by Dunja Jogan.

\section{Circle sums and a bordism invariant}

\subsection{Circle sums and $\Z/2$-equivariant $Pin^+$-structures}{\label{Section CSP}}  The basic cut-and-paste construction along codimension three submanifolds that is used in this paper is defined as follows. 

\begin{definition}{\label{Definition CS}} \cite{[CS], [HKT]}. Let $M_1, M_2$ be non-orientable closed smooth 4-manifolds that both admit a $Pin^{+}$-structure. Denote the non-trivial $D^3$-bundle over $S^1$ by  $D^3 \tilde{\times} S^1$, and fix a $Pin^{+}$-structure on it. Let \begin{equation}i_i: D^3 \tilde{\times} S^1 \hookrightarrow M_i\end{equation} for $i = 1, 2$ be smooth embeddings that represent a non-trivial orientation reversing element in the group $\pi_1(M_i)$ such that $i_1$ preserves the $Pin^{+}$-structure and $i_2$ reverses it. Define by\begin{equation}M_1 \#_{S^1} M_2: = (M_1 - i_1(D^3\tilde{\times}S^1)) \cup  (M_2 - i_2(D^3 \tilde{\times} S^1))\end{equation}

the \emph{circle sum} of $M_1$ and $M_2$, which admits a $Pin^+$-structure $\phi_1 \#_{S^1} - \phi_2$.
\end{definition}

The choices of $Pin^{+}$-structures are parametrized by $H^1(M_1\#_{S^1}M_2; \Z/2\Z)$ \cite{[KT]}, \cite[Corollary 6.4]{[S]}. Since they are submanifolds of codimension three, any two homotopic embeddings of a loop into a 4-manifold are isotopic \cite[Example 4.1.13]{[GS]}. Moreover, given group isomorphisms $\pi_1(M_1) \cong \Z$ and $\pi_1(M_2) \cong H$, the Seifert-van Kampen theorem implies $\pi_1(M_1\#_{S^1} M_2) \cong H$. Similarly, if $\pi_1(M_1) \cong \Z/2 \cong \pi_1(M_2)$, then  $\pi_1(M_1\#_{S^1} M_2) \cong \Z/2\Z$. Regarding the bordism class of the resulting manifold, we have the following result.

\begin{proposition}{\label{Proposition Bord}} Let $(M_i, \phi_i)$ be a closed non-orientable manifold with $Pin^+$- structure $\phi_i$ for $i = 1, 2$. The circle sum $M_1\#_{S^1} M_2$ and the connected sum $M_1\# M_2$ are $Pin^+$-bordant to the disjoint union $(M_1, \phi_1) \sqcup (M_2, \phi_2)$.
\end{proposition}

\begin{proof} The case of connected sums is proven in \cite[Section 7]{[S]}. To prove the claim for circle sums, we use an argument in \cite[p. 160]{[S]}. Denote by $\rho_i\subset M_i$ for $i = 1, 2$ the loops involved in the construction of the circle sum of Definition \ref{Definition CS}, and let $\nu(\rho_i)$ be its tubular neighborhood. In particular, $\nu(\rho_1) = \nu(\rho_2)$, and we denote both by $\nu(\rho)$. The claimed bordism is
\begin{equation}{\label{Bordism}}
M_1\times [-1, 1] \cup_{\nu(\rho)} D(\rho)\times [-1, 1] \cup_{\nu(\rho)} M_2\times [-1, 1].
\end{equation}
\end{proof}

The $\eta$-invariant can be computed using the following notion as it is done in Lemma \ref{Lemma Values}.

\begin{definition}{\label{Definition Equiv}} \cite[p. 153]{[S]}. $\Z/2$-equivariant $Pin^+$-structure. Let $M$ be a non-orientable $n$-manifold, and suppose its orientation double cover $\widetilde{M}$ is the boundary of an $(n + 1)$-manifold $W$ with involution $T: W\rightarrow W$ extending the non-trivial covering transformation of $\widetilde{M}\rightarrow M$, and whose fixed points are isolated. Suppose $M$ has a $Pin^+$-structure $\phi$, which determines a $Pin^+$-structure $\widetilde{\phi}$ on $\widetilde{M}$.

A $\Z/2$-equivariant $Pin^+$-structure on $W$ is a $Pin^+$-structure $(P, \phi_W)$ on the tangent bundle $TW$ together with a $Pin^+(n + 1)$-equivariant involution $\overline{T}:P\rightarrow P$ making the diagram

\begin{center}
$
\begin{CD}
TW @>f>\cong> P\times_{Pin^+(n + 1)} \R^{n + 1}\\
@VdTVV   @VV\overline{T}\times idV\\
TW @ >f>\cong> P\times_{Pin^+(n + 1)} \R^{n + 1}
\end{CD}
$
\end{center}

commutative. In particular, $\phi_W$ determines $\widetilde{\phi}$. The data of a $\Z/2$-equivariant $Pin^+$-structure is encoded in the quadruple $(W, P, f, \overline{T})$.\\

If $x\in W$ is a fixed point of $T$, then $dT: T_xW \rightarrow T_xW$ is multiplication by $-1$. Identifying $P_x$ with $Pin^+(n + 1)$, this implies that $\overline{T}: P_x\rightarrow P_x$ is left-multiplication by $s_{n + 1} = e_1 \cdot \ldots \cdot e_{n + 1}$ or $-(s_{n + 1})$. Here, $s_{n + 1}$ is the central element in the Clifford algebra $C^+(\R^{n+1})$ that squares to the identity (see \cite[Section 3]{[S]} for details). An index $\iota_x\in \{\pm\}$ is attached to each fixed point according to the sign of the involution $\overline{T}$.

\end{definition}

\subsection{Mapping tori, exotic $\R P^4$'s and their universal covers}{\label{Section FRP}} Circle sums with mapping tori were used in \cite{[CS]} to unveil inequivalent smooth structures on $\R P^4$. We now recall the construction with the purpose of making this note as self-contained as possible.  Take the element in the group $GL(3; \Z)$ given by
\[A =\left( \begin{array}{ccc}
0 & 1  & 0   \\
0 & 0  & 1  \\
- 1 & u  & v 
\end{array} \right)\]

that satisfies $det A = -1$ and $det(I - A^2) = \pm 1$. Considering the 3-torus as $T^3 = \R^3/\Z^3$, the choice of matrix yields a diffeomorphism $\varphi_{A}: T^3\rightarrow T^3$ with induced map $(\varphi_{A})^{\ast} = A$ on $H_1(T^3; \Z) = \Z^3$. The mapping torus $M_A$ of $\varphi_A$, i.e., the quotient of $T^3\times [0, 1]$ under the identification $(x, 0) \sim (\varphi_A(x), 1)$, is a $T^3$-bundle over $S^1$.\\

Let $X$ be a closed smooth non-orientable 4-manifold with fundamental group of order two, let $\rho \subset X$ be the loop representing an element in the group $\pi_1(X)$ that reverses orientation. Denote its tubular neighborhood by $\nu(\rho)$. Let $M_0$ be the tubular neighborhood of the loop $(x_0 \times [0, 1]/\sim)$ $\subset M_A$, where $\varphi_A(x_0) = x_0$ for $x_0 \in T^3$. Construct the circle sum  \begin{equation}{\label{Construction}}
X_A:= X \#_{S^1} M_A = (X - \nu(\rho)) \cup_{\partial \nu(\rho)} (M_A - M_0)
\end{equation}

of Definition \ref{Definition CS} by gluing together $X - \nu(\rho)$ and $M_0$ along their common boundary. The following result shows that the homeomorphism class of $X$ is invariant under this cut-and-paste operation.

\begin{lemma}{\label{Lemma Homeo}} \cite[Section 3]{[CS]}, \cite[Section 7]{[S]}. The 4-manifolds $X$ and $X_A$ are homeomorphic.
\end{lemma}

\begin{proof} The manifolds $X$ and $X_A$ are homotopy equivalent, and \cite[Theorem 3.1]{[CS]} implies that there exists a simple homotopy equivalence; the Whitehead group $Wh(\Z/2\Z)$ is trivial (see \cite[Proof of Theorem 13 A.1]{[W]}). The surgery exact sequence
\begin{equation}
\longrightarrow L_5(\Z \pi_1(X), \omega_1(X)) \longrightarrow \mathcal{S}^{TOP}(X) \longrightarrow [X, G/TOP] \longrightarrow
\end{equation}

\cite[Theorem 10.3, Theorem 10.5]{[W]} implies the existence of an $s$-cobordism between these manifolds. Since the surgery group vanishes \cite[Theorem 13.1]{[Wa]}, it follows from \cite{[F]} that these manifolds are homeomorphic. 

One can alternatively appeal to the topological classification in \cite[Theorem 3]{[HKT]} to conclude the existence of the homeomorphism. 

\end{proof}

\begin{convention}{\label{Convention 1}} Through out the paper, we will denote by $Q$ the inequivalent smooth structure on $\R P^4$ built in \cite{[CS]} using the mapping torus that corresponds to the matrix with entries $u = 1$ and $v = 0$. Its universal cover was proven to be diffeomorphic to $S^4$ in \cite{[G0]}. \end{convention}

\subsection{Bordism and diffeomorphism type invariant: the $\eta$-invariant}{\label{Section eta1}}  Let $M$ be a 4-manifold with a Riemannian metric $g$, and a $Pin^+$-structure $\phi$. The fundamental invariant of this paper is the $\eta$-invariant, $\eta(M, g, \phi)$, and we refer the reader to \cite{[Gi], [S]} for details.  It is proven in  \cite[Proposition 4.3]{[S]} that $\eta(M, g, \phi)$ mod $2\Z$ is a $Pin^+$-bordism invariant, i.e., it depends on the $Pin^+$-bordism class of $(M, \phi)$ but not on the choice of Riemannian metric. Thus, we drop $g$ from our notation, and we denote it by\begin{equation} \eta(M, \phi).\end{equation} 

Following a suggestion in \cite{[Gi]}, the known inequivalent smooth structures on $\R P^4$ \cite{[CS], [FS]} were detected using this spectral invariant in \cite{[S], [O]}. The value $\eta(M, \phi)$ mod $2\Z$ determines $Pin^+$-bordism classes $[M, \phi]\in \Omega_4^{Pin^+}\cong \Z/16\Z$, where addition in this group is the circle sum $\#_{S^1}$ of Definition \ref{Definition CS} and $\R P^4$ is the generator \cite{[KT]}.\\

The calculation of the $\eta$-invariant in \cite{[S], [O]} shows that there is no $Pin^+$-structure preserving diffeomorphism between the real projective 4-space and $Q$, hence telling apart the smooth structure of the exotic $\R P^4$'s \cite{[CS], [FS]} from the standard one. The precise values are collected in the following result.

\begin{theorem}{\label{Theorem Stolz}} \cite[Theorem 3.3]{[Gi]}, \cite[Theorem A]{[S]}, \cite[Theorem A]{[O]}. 
\begin{center} $\eta(\R P^4, \phi) = \pm 1/8$ mod $2\Z$ and $\eta(Q, \phi') = \pm 7/8$ mod $2\Z$ \end{center} for all $Pin^{+}$-structures $\phi, \phi'$ on $\R P^4$ and $Q$ respectively.
\end{theorem}

The change of the $\eta$-invariant under the surgical procedure of Section \ref{Section FRP} is \begin{equation}
\eta(M\#_{S^1} M_A, \phi') = \eta(M, \phi) + \eta(M_A, \phi_A) = \eta(M, \phi) + 1
\end{equation} mod $2\Z$ for the corresponding $Pin^+$-structures \cite[Proposition 7.3]{[S]}, which yields the following general result. 

\begin{theorem}{\label{Theorem Stolz2}} \cite[Theorem 7.4]{[S]}. Let $M$ be a closed smooth non-orientable 4-manifold that admits a $Pin^+$-structure $\phi$. Suppose $H^1(M; \Z/2\Z) \cong \Z/2\Z$, and $\eta(M, \phi) \neq \pm 1/2$ mod $2\Z$. Then, \begin{equation} \eta(M\#_{S^1} M_A, \phi') \neq \eta(M, \phi) \end{equation} for all  $Pin^+$-structures on $M\#_{S^1} M_A$.
\end{theorem}

We finish this section with the following observation regarding the limitations of using the $\eta$-invariant to discern inequivalent smooth structures, which was kindly pointed out to us by a referee.

\begin{proposition}{\label{Proposition Limits}} The $\eta$-invariant distinguishes at most two inequivalent smooth structures on closed non-orientable smooth Pin$^+$-manifolds with fundamental group $\Z/2$.
\end{proposition}

\begin{proof} Let $M$ and $M'$ be homeomorphic closed non-orientable 4-manifolds with fundamental group $\Z/2$ that admit a $Pin^+$-structure. By \cite[Theorem 1]{[HKT]} we have $\eta(M, \phi) = \eta(M', \phi')$ mod $\Z$ for $Pin^+$-structures on $M$ and $M'$ respectively. Equivalently, $\eta(M, \phi) = \eta(M', \phi') + n$ for $n\in \Z$. Suppose now that  $\eta(M, \phi) \neq \eta(M', \phi')$ mod $2\Z$. It implies that $n$ is an odd integer, and in particular $\eta(M, \phi) = \eta(M', \phi') + 1$. Since the $\eta$-invariant mod $2\Z$ is a diffeomorphism invariant, it distinguishes at most two inequivalent smooth structures on such 4-manifolds.

\end{proof}

\subsection{Computations of the $\eta$-invariant}{\label{Section eta2}} We now engage in the calculations of the $\eta$-invariant that are required to prove Theorem \ref{Theorem Pr} and Theorem  \ref{Theorem M}. The bordism constructed in Proposition \ref{Proposition Bord} between circle sums, connected sums, and disjoint unions allows us to compute the invariant using the following result. 

\begin{proposition}{\label{Proposition Ad}}  

\begin{itemize}
\item \cite[Section 7]{[S]}. The $\eta$-invariant is additive modulo $2\Z$ with respect to connected sums, and circle sums.

\item \cite[Corollary 5.2]{[S]}. Let $M$ be a closed 4-manifold with Spin-structure $\phi$. Then
\begin{center}
$\eta(M, \phi) = 1/16$ $\sigma(M)$ mod $2\Z$
\end{center}

where $\sigma(M)$ is the signature of $M$.
\end{itemize}
 
\end{proposition}

The following proposition provides calculations of the $\eta$-invariant mod $2\Z$ for the respective $Pin^+$-structures, and for all $k\in \N$.

\begin{proposition}{\label{Proposition Values}} \begin{equation}{\label{Item 1}} \eta(S(2\gamma \oplus \R)\# (k - 1)(S^2\times S^2), \pm \phi) = 0  \end{equation} \begin{equation}{\label{Item 2}} \eta(\#_{S^1}r\cdot \R P^4 \# (k - 1)(S^2\times S^2), \pm \phi) = \pm \frac{r}{8} \end{equation} \begin{equation}{\label{Item 3}} \eta(S^3\widetilde{\times} S^1 \# (k - 1)(S^2\times S^2), \pm \phi_{k}) = 0\end{equation} \begin{equation}{\label{Item 4}} \eta(A \# (k - 1)(S^2\times S^2), \pm \phi_{A, k}) = \pm 1.\end{equation}
\end{proposition}

\begin{proof} We first proof Item \ref{Item 1} of the proposition. The sphere bundle of the Whitney sum of two copies of the real Hopf bundle and the trivial bundle is the boundary of a five dimensional disk bundle that has a $Pin^+$-structure. Therefore, $S(2\gamma \oplus \R)$ represents the zero element in $\Omega^{Pin^+}_4$. Consider the decomposition \begin{equation}\R P^4 = (D^2\widetilde{\times} \R P^2) \cup (D^3 \widetilde{\times} S^1)\end{equation} as the union of the twisted 2-disk bundle over the real projective plane and the twisted 3-disk bundle over the circle. The $S^2$-bundle over $\R P^2$ is the double 
\begin{equation} S(2\gamma \oplus \R) = (D^2 \widetilde{\times} \R P^2) \cup (D^2 \widetilde{\times} \R P^2).\end{equation}
Denote by $\overline{\R P^4}$ the real projective 4-space endowed with the $Pin^+$-structure $- \phi$; $\overline{\R P^4}$ represents the bordism class $[\R P^4, -\phi] = - [\R P^4, \phi] \in \Omega_4^{Pin^+}$. The double is the circle sum $S(2\gamma \oplus \R) = \R P^4 \#_{S^1} \overline{\R P^4}$, whose $\eta$-invariant is zero for both $Pin^+$-structures. The claim now follows from Proposition \ref{Proposition Ad}, since the signature of the connected sum of copies of $S^2\times S^2$ is zero. 

The proof of Item \ref{Item 2} follows from Theorem \ref{Theorem Stolz} and Proposition \ref{Proposition Ad}.

The proof of Item \ref{Item 3} goes as follows. Notice that $H^1(S^3\widetilde{\times}S^1; \Z/2) \cong \Z/2$, and the twisted $S^3$-bundle over $S^1$ admits two $Pin^+$-structures that we denote by $\pm \phi$. Moreover, $\eta(S^3\widetilde{\times} S^1, + \phi) = - \eta(S^3\widetilde{\times} S^1, - \phi)$ mod $2\Z$. Proposition \ref{Proposition Ad} implies that mod $2\Z$ \begin{equation}\eta(S^3\widetilde{\times} S^1 \# (k - 1)(S^2\times S^2), \phi_k) = 
\eta(S^3\widetilde{\times} S^1, \phi). \end{equation} We claim that $\eta(S^3\widetilde{\times} S^1, \pm \phi) = 0$ mod $2\Z$. 
To compute the value of the $\eta$-invariant, take the diffeomorphism $\R P^4 \#_{S^1} S^3\widetilde{\times} S^1 \cong \R P^4$ of the circle sum of $\R P^4$ and $S^3\widetilde{\times}S^1$. Proposition \ref{Proposition Ad} implies that mod $2\Z$ \begin{equation}\eta(\R P^4 \#_{S^1} S^3\widetilde{\times} S^1, \phi'') = \eta(\R P^4, \phi_{\R P^4}) + \eta(S^3\widetilde{\times} S^1, \phi) = \pm 1/8. \end{equation}

Thus, $\eta(S^3\widetilde{\times} S^1, \pm \phi) = 0$ mod $2\Z$ as claimed.

The proof of Item \ref{Item 4} is as follows. The manifold $A$ is constructed by blowing down an $\R P^2$ in $Q\#S^2\times S^2$ \cite[Section 0]{[AO]}, i.e., \begin{equation}A = ((Q\# S^2\times S^2) - \nu(\R P^2)) \cup (D^3 \widetilde{\times} S^1),\end{equation} where $Q$ is the exotic $\R P^4$ of \cite{[CS]} (see Convention \ref{Convention 1}), $\nu(\R P^2) = D^2\widetilde{\times} \R P^2$ is the tubular neighborhood of the real projective plane inside $Q\# S^2\times S^2$, and $D^3\widetilde{\times} S^1$ is the non-orientable $D^3$-bundle over $S^1$. We use an argument in \cite[p. 160]{[S]} to calculate the value mod $2\Z$ of the invariant, since $(A, \phi_A)$ is $Pin^+$-bordant to the disjoint union of $(Q\#S^2\times S^2, \phi_{Q\#S^2\times S^2})$, $(\R P^4, \phi_{\R P^4})$ and $(S^3\widetilde{\times} S^1, \phi)$ with the appropriate $Pin^+$-structures. Using the decomposition $\R P^4 = D^2\widetilde{\times} \R P^2 \cup D^3 \widetilde{\times} S^1$, a bordism is given by \begin{equation}
(Q\# S^2\times S^2 \times [0, 1]) \cup_{\nu(\R P^2)} (\R P^4\times [0, 1]) \cup_{\nu(\gamma)} (S^3\widetilde{\times} S^1 \times [0, 1]).\end{equation} We are denoting by $\nu(\gamma) = D^3\widetilde{\times} S^1$ is a tubular neighborhood of the loop $\gamma \subset S^3 \widetilde{\times} S^1$ that carries the generator of the fundamental group $\pi_1(S^3 \widetilde{\times} S^1) = \Z$, and the boundaries are $\partial(\nu(\R P^2)) = S^2\widetilde{\times} S^1 = \partial (D^3 \widetilde{\times} S^1)$.

Proposition \ref{Proposition Ad}, Theorem \ref{Theorem Stolz}, and Item \ref{Item 3} of this proposition allow us to conclude now that $\eta(A, \pm \phi_A) = \pm 1$ mod $2\Z$.
\end{proof}

Let us now consider non-orientable 4-manifolds $X_1$ and $X_2$ with fundamental group $\pi_1(X_i) \cong \Z/2$ and second Stiefel-Whitney class $\omega_2(X_i) = 0$ for $i = 1, 2$. There are four $Pin^+$-structures on the connected sums, since there is an isomorphism $H^1(X_1\# X_2; \Z/2) \cong \Z/2 \oplus \Z/2$. All four $Pin^+$-structures on the connected sum can be constructed in terms of the two $Pin^+$-structures on $X_i$ (cf. \cite{[KT]}). In particular, there are four values of the $\eta$-invariant to consider in order to use it to discern smooth structures on these connected sums. Using Proposition \ref{Proposition Bord} and Proposition \ref{Proposition Ad} we obtain the following result regarding these values.

\begin{proposition}{\label{Proposition Comp}} Let $X_i$ $i = 1, 2$ be non-orientable 4-manifolds with $\pi_1(X_i) \cong \Z/2$, and $\omega_2(X_i) = 0$. Suppose
\begin{center}
$\eta(X_1, \pm \phi_1) = \pm r_1/s_1$ mod $2\Z$

$\eta(X_2, \pm \phi_2) = \pm r_2/s_2$ mod $2\Z$
\end{center}

for $r_i, s_i \in \Z$, and $Pin^+$-structures $\pm \phi_1$, $\pm \phi_2$ on $X_1$ and $X_2$, respectively. The values of the $\eta$-invariant mod $2\Z$ on the connected sum $X_1\#X_2$ are \begin{equation}\eta(X_1\# X_2, \pm \phi_1\# \pm \phi_2) =  \pm r_1/s_1 \pm r_2/s_2\end{equation} \begin{equation}\eta(X_1\# X_2, \pm \phi_1 \# \mp \phi_2) = \pm r_1/s_2 \mp r_2/s_2.\end{equation}
\end{proposition}

Yet another interesting result of \cite{[S]} its the following topological method that can be used to compute the spectral invariant using the structure introduced in Definition \ref{Definition Equiv}.

\begin{proposition}{\label{Proposition Inv}} \cite[Section 3]{[S]} Let $M$ be a 4-manifold with Riemannian metric $g$, $Pin^+$-structure $\phi$, and a $\Z/2$- equivariant $Pin^+$-structure $(W, P, f, \overline{T})$. Then,
\begin{center}
$\eta(M, g, \phi) = \frac{1}{8} \sum \iota_x$ mod $2\Z$
\end{center}

where $\iota_x$ is the index of Definition \ref{Definition Equiv}, and the sum extends over all fixed points of $\overline{T}$.

\end{proposition}

The following lemma summarizes computations of the $\eta$-invariant mod $2\Z$ for the corresponding $Pin^+$-structures that are needed to prove Theorem \ref{Theorem M}.

\begin{lemma}{\label{Lemma Values}} 
\begin{equation}{\label{Item L1}}\eta(S^3\widetilde{\times} S^1, \pm \phi) = 0\end{equation}\begin{equation}{\label{Item L2}} \eta(Kb\times S^2, \phi_i) = 0\end{equation}\begin{equation}{\label{Item L3}} \eta(\xi_3, \phi_i') = 0\end{equation}
\begin{equation}{\label{Item L4}}\eta(Kb\times T^2, \phi_j') = 0\end{equation} for $i = 1, 2, 3, 4$ and $j = 1, 2, \ldots, 15, 16$.
\end{lemma}

\begin{proof} We begin with the proof of Item \ref{Item L1}, which has already been proven in Proposition \ref{Proposition Values}. We give here a different proof using $\Z/2$-equivariant $Pin^+$-structures. The quadruple of Definition \ref{Definition Equiv} is given as follows. The orientation  double cover $S^3\times S^1$ is the boundary of $W:= S^3\times D^2$ with tangent bundle $T(S^3\times D^2) = S^3 \times D^2 \times D^5$. There is a natural vector bundle isomorphism between this bundle and $E(\sigma^+)$. Consider the involution $T: S^3\times D^2 \longrightarrow S^3\times D^2$ given by $(x, y)\mapsto (r(x), -y)$, where $r: S^3\longrightarrow S^3$ is reflection in a hyperplane cf. \cite[p. 2]{[JK]}. This involution can be extended to an involution $\overline{T}$ on the principal bundle $P = S^3\times D^2 \times Pin^+(5)$ by multiplying from the left by $s_5 = e_1\cdots e_5$ on the third factor. The involution $\overline{T}$ has no fixed points. Proposition \ref{Proposition Inv} implies $\eta(S^3\widetilde{\times} S^1, \phi) = 0$ mod $2\Z$. Moreover, $\eta(S^3\widetilde{\times} S^1, \phi) = -  \eta(S^3\widetilde{\times} S^1, - \phi)$ mod $2\Z$, and the claim follows.
\smallskip\smallskip

We now prove Item \ref{Item L2}. There are four choices of $Pin^+$-structures to be considered given that $H^1(Kb\times S^2; \Z/2) \cong \Z/2\oplus \Z/2$; they arise from the four $Pin^+$-structures on the Klein bottle, and the one on the 2-sphere. We abuse notation and denote by $\phi_i$ the $Pin^+$-structure on $Kb\times S^2$ that arises from the corresponding structure $\phi_i$ on $Kb$ for $i = 1, 2, 3, 4$.

Consider $Kb$ as the twisted $S^1$-bundle over $S^1$, $S^1\widetilde{\times} S^1$. Two of the $Pin^+$-structures, say $\phi_1$ and $\phi_2$, extend over the twisted $D^2$-bundle over $S^1$, $D^2\widetilde{\times} S^1$ \cite[p. 206]{[KT]}. Hence, $(Kb\times S^2, \phi_1)$ and $(Kb\times S^2, \phi_2)$ are $Pin^+$-boundaries of $D^3\widetilde{\times} S^1 \times S^2$, and $\eta(Kb\times S^2, \phi) = 0$ mod $2\Z$ for $i = 1, 2$. Although the other two $Pin^+$-structures on $Kb$, $\phi_3$ and $\phi_4$, do not $Pin^+$-bound, they are $Pin^+$-cobordant to each other \cite[p. 206]{[KT]}. In this case, $(Kb\times S^2, \phi_3)$ is the $Pin^+$-boundary of $Kb\times D^3$, and we conclude $\eta(Kb\times S^2, \phi) = 0$ mod $2\Z$ for $i = 3, 4$.

Analogously, Proposition \ref{Proposition Inv} can be used to calculate the $\eta$-invariant as follows. Take $(Kb\times S^2, \phi_3)$ with orientation double cover  $S^1\times S^1 \times S^2 = S^1\times S^1 \times S^2$ bounding $W = S^1\times S^1\times D^3$. There is a natural vector bundle isomorphism between $T(S^1\times S^1\times D^3)$ and $E(\sigma^+)$. The involution $T$ is given as reflection on the first circle factor, rotation by 180 degrees on the second circle factor, and the identity on the third 3-disk factor. This involution is the canonical product extension that arises from the 2-cover $T^2\rightarrow Kb$, and it has no fixed points. The involution $T$ can be extended to $P = S^1\times S^1 \times D^3 \times Pin^+(5)$ by left-multiplication with $s_{5} = e_1 \cdot \ldots e_5$ on the $Pin^+(5)$ factor, thus obtaining the involution $\overline{T}$ of the quadruple in Definition \ref{Definition Equiv}. Since $\overline{T}$ has no fixed points, Proposition \ref{Proposition Inv} implies $\eta(Kb\times S^2, \phi_3) = 0$ mod $2\Z$.

\smallskip\smallskip

Regarding Item \ref{Item L3}, the computation of the $\eta$-invariant for total space $\xi_3$ of the non-trivial non-orientable $S^2$-bundle over $Kb$ that admits a $Pin^+$-structure is analogous to the one for $Kb\times S^2$ given in Item \ref{Item L2}. To appeal to Proposition \ref{Proposition Inv}, one uses the description of $\xi_3$ given in \cite[Proof of Theorem 10.11]{[H]} that we now recall. Consider the elements $(A, \beta, C)$ of the group $Isom(S^2\times \R^2) = O(3) \times (\R^2 \rtimes O(2))$ such that an isometry sends $(v, x)\in S^2\times \R^2$ to $(Av, Cx + \beta)$, and let $\vec{i} = (1, 0)$ and $\vec{j} = (0, 1)$. The manifold $\xi_3$ is realized as a quotient using the isometries $(A, \frac{1}{2}\vec{i}, C_A)$ and $(B, \vec{j},C_B)$ with $A = C_B = I$ the identity, 
 \[ C_A=\left( \begin{array}{cc}
1 & 0  \\
0 & -1 
\end{array}\right),\] 

and $B$ is the antiholomorphic involution $R: \mathbb{CP}^1 \rightarrow \mathbb{CP}^1$ that assigns $z \mapsto - \overline{z}$ for $z\in \mathbb{CP}^1$ once the usual identification $S^2 = \C \cup \{\infty\} = \mathbb{CP}^1$ takes place. From the description of $\xi_3$ we see that the involution $T: D^3\times S^1\times S^1$ has no fixed points. 

\smallskip\smallskip

We now prove Item \ref{Item L4}. There are sixteen choices of $Pin^+$-structures to be considered, since $H^1(Kb\times T^2; \Z/2) \cong \Z/2\oplus \Z/2\oplus \Z/2 \oplus \Z/2$, and they arise from the four $Pin^+$-structures on the Klein bottle $\{\phi_1, \phi_2, \phi_3, \phi_4\}$ and the four $Pin^+$-structures on the 2-torus $\{\phi_1', \phi_2', \phi_3', \phi_4'\}$. We denote the $Pin^+$-structure on $Kb\times T^2$ by $\phi_l\times \phi_m'$ for $l, m = 1, 2, 3, 4$.

As it was shown in the proof of Item \ref{Item L1}, $\eta(Kb\times T^2, \phi_l\times \phi_m') = 0$ mod $2\Z$ for $i = 1, 2$ and $m = 1, 2, 3, 4$, since $Kb\times T^2$ is the $Pin^+$-boundary of $D^2\widetilde{\times} S^1\times T^2$ for these eight structures. In more generality, since both $Pin^+$-structures on the circle $Pin^+$-bound \cite[p. 206]{[KT]}, so do all $Pin^+$-structures on $T^2 = S^1\times S^1$, and the $\eta$-invariant mod $2\Z$ vanishes for all choices of $Pin^+$-structures.
\end{proof}

\section{Construction of inequivalent smooth structures on non-orientable 4-manifolds}{\label{Section Exotica}}

In this section, we prove Theorem \ref{Theorem 0}, Theorem \ref{Theorem M}, Corollary \ref{Corollary M}, as well as Theorem \ref{Theorem Pr} modulo the claim regarding Gluck twists. Constructions of inequivalent smooth structures on non-orientable 4-manifolds can be found in \cite{[CS], [FS], [AO], [FS1], [K], [AT], [S]}.

\subsection{Proof of Theorem \ref{Theorem 0}}{\label{Section 10}} The result is a corollary of the work done in \cite{[S]} that we have cited in Theorem \ref{Theorem Stolz2} and Proposition \ref{Proposition Values}. Let $M$ be either a circle sum $\#_{S^1}r\cdot \R P^4$ for $r = 1, 2, 3$ or the 2-sphere bundle over the real projective plane $S(2\gamma \oplus \R)$.The inequivalent smooth structure is built as the circle sum $M\#_{S^1} M_A$ described in Section \ref{Section FRP}. Lemma \ref{Lemma Homeo} shows that the construction does not change the homeomorphism class of the manifold. Proposition \ref{Proposition Values} now says that the values of the $\eta$-invariant are different from $\pm \frac{1}{2}$ for these manifolds, and Theorem \ref{Theorem Stolz2} applies. Proposition \ref{Proposition Ad} implies that smooth structures remain inequivalent under connected sums with an arbitrary number of copies $S^2\times S^2$.\\

The claim regarding the existence of a diffeomorphism after taking the connected sum with a single copy of $\mathbb{CP}^2$ is a corollary of \cite[Theorem 2]{[AO]} as we now explain. The analysis of the handlebody structure of $Q$ \cite[Section 2]{[AO]} indicates that it can be decomposed as $N\cup D^3\widetilde{\times} S^1$, where $N$ is an exotic $D^2\widetilde{\times} \R P^2$. Moreover, $N\# \mathbb{CP}^2$ is diffeomorphic to $D^2\widetilde{\times} \R P^2\# \mathbb{CP}^2$ \cite[Section 3]{[AO]}. There is a diffeomorphism $(\R P^4 \#_{S^1} \R P^4) \#_{S^1} M_A$ and $\R P^4 \#_{S^1} Q$, where the latter circle sum glues $N$ to the complement of the twisted $D^3$ bundle over $S^1$ inside $\R P^4$. Hence, $(\#_{S^1}2\cdot \R P^4\#_{S^1} M_A) \# \mathbb{CP}^2$ is diffeomorphic to $(\#_{S^1} 2\cdot \R P^4) \# \mathbb{CP}^2$. The claim for the other circle sums follows from induction on the copies of $\R P^4$ used in the circle sum. The claim for the 2-sphere bundle over the real projective plane also follows from \cite{[AO]} once it expressed as the circle sum $\R P^4 \#_{S^1} \overline{\R P^4}$ as in was done for the proof of Proposition \ref{Proposition Values}. 

\subsection{Item (1) of Theorem \ref{Theorem M}, Corollary \ref{Corollary M}, and Theorem \ref{Theorem Pr}} Let us begin by showing Item (1) of Theorem \ref{Theorem M}. As it was mentioned in the introduction, it has been shown in \cite{[AO]} that for every $k\in \N$, there exists a manifold $A_k$ that is homeomorphic but not diffeomorphic to \begin{equation}S^3\widetilde{\times} S^1 \# k(S^2\times S^2),\end{equation} and we labelled $A:=A_1$. Proposition \ref{Proposition Values} shows that the values of the $\eta$-invariant mod $2\Z$ for the smooth structures are indeed different for both $Pin^+$-structures.\\

Corollary \ref{Corollary M} follows from \cite{[AO]}, using the homeomorphism classification in \cite{[W], [QK]}. Indeed, it is shown in \cite[Theorem 2]{[W]} that any closed smooth non-orientable 4-manifold $M$  with infinite cyclic fundamental group and $\omega_2(M) = 0$ is stably homeomorphic to a connected sum of $S^3\widetilde{\times} S^1$ with copies of $S^2\times S^2$. That is, there is a homeomorphism \begin{equation}M \# r(S^2\times S^2) \cong_{C^0} S^3\widetilde{\times} S^1 \# r(S^2\times S^2).\end{equation}The minimal number of stabilizations is $r = 3$ by \cite{[QK]} (cf. \cite{[CSixt]}). This is satisfied by manifolds with Euler characteristic at least six.\\

Regarding the proof of Theorem \ref{Theorem Pr} modulo the claim concerning Gluck twists, we have the following. Let $X$ be a closed smooth non-orientable 4-manifold with a $Pin^+$-structure. The manifold $Y$ is the circle sum \begin{equation} Y:= X\#_{S^1} A,\end{equation} where $A$ is the exotic manifold constructed in \cite{[AO]} (see Section \ref{Section 10}). The manifold $Y$ is homeomorphic to $X\# (S^2 \times S^2)$ by construction, since $A$ is homeomorphic to $S^3\widetilde{\times} S^1 \# S^2\times S^2 = (D^3\widetilde{\times} S^1 \cup D^3\widetilde{\times} S^1) \# S^2\times S^2$. Proposition \ref{Proposition Ad} and Proposition \ref{Proposition Values} imply the claim regarding the value of the $\eta$-invariant of $Y$. The claim regarding Gluck twists is proven in Section \ref{Section Gluck}.

\subsection{Proof of Theorem \ref{Theorem M}} The inequivalent smooth structures of Items (1) - (4) of Theorem \ref{Theorem M} are constructed as an application of Theorem \ref{Theorem Pr} to circle sums $Y_k:= X\#_{S^1} A_k$, where $A_k$ is the exotic manifold constructed in \cite{[AO]} (see Section \ref{Section 10}) for $k\in \N$, and $X$ is either $Kb\times S^2, \xi_3$ or $Kb\times T^2$. The inequivalent smooth structure of Item (5) is constructed using those of Theorem \ref{Theorem 0}. Proposition \ref{Proposition Values}, Lemma \ref{Lemma Values}, and Proposition \ref{Proposition Comp} imply the following result.

\begin{proposition} There exists a manifold that is homeomorphic but not diffeomorphic to \begin{equation} Kb\times S^2\# k(S^2\times S^2),\end{equation}\begin{equation}\xi_3 \# k(S^2\times S^2),\end{equation}\begin{equation}T^2\times Kb\#k(S^2\times S^2),\end{equation} and \begin{equation} X_1\#X_2\end{equation}

for every $k\in \N$. The smooth structures are discerned by the $\eta$-invariant.
\end{proposition}

\begin{remark} There are infinitely many non-orientable closed topological 4-manifolds homotopy equivalent to a connected sum that are not homeomorphic to a trivial connected sum as in Item (5) of Theorem \ref{Theorem M} \cite{[C1], [JK1], [BDK]}. It is proven in \cite[Theorem 3.4]{[QK]} (cf. \cite{[CSixt]}) that there exists a unique homeomorphism class after stabilization with at least three copies of $S^2\times S^2$, i.e., for \begin{equation}X_1\# X_2 \# 3(S^2\times S^2). \end{equation}
\end{remark}

\section{Non-orientable handlebodies and Gluck twists}{\label{Section Hand}}

We build greatly upon \cite{[AO], [A], [AL]} to prove the results in this section.

\subsection{Handlebody of exotic manifolds of Theorem \ref{Theorem 0}} 
The handlebody structure of an exotic $\R P^4$ was analyzed in \cite{[AO]}, where it is shown that it decomposes as the circle sum of an exotic 2-disk bundle over $\R P^2$ and $D^3 \widetilde{\times} S^1$. The construction of exotic manifolds of Theorem \ref{Theorem 0} can be expressed as blowing up an $\R P^2$ as in \cite[Section 0]{[AO]}, thus we construct their handlebodies building greatly upon \cite{[AO]}. The handlebody of $S(2\gamma \oplus \R) = D^2 \widetilde{\times} \R P^2 \cup D^2 \widetilde{\times} \R P^2$ is given in Figure 1 on page 14. The same figure with the $p_2$- and $q_2$- framed 2-handle removed yields a handlebody of $D^2 \widetilde{\times} \R P^2$, provided $p_i + q_i$ is odd \cite[Section 0]{[AO]}.\\

A handlebody of an exotic copy of $S(2\gamma \oplus \R)$ is given in Figure 2 on page 14, which is constructed by turning the handlebody of $D^2\widetilde{\times} \R P^2$ upside down and adjoining it to the handlebody of the exotic 2-disk bundle over the real projective plane constructed in \cite[Figure 1.33]{[AO]}. Different choice of gluing map between the common boundary of $D^2\widetilde{\times} \R P^2$ and its exotic copy yields a handlebody for the inequivalent smooth structures on the manifolds $\#_{S^1} r \cdot \R P^4$.

\begin{figure}{\label{Figure 1}}
\begin{center}
\includegraphics[viewport= 430 150 150 730]{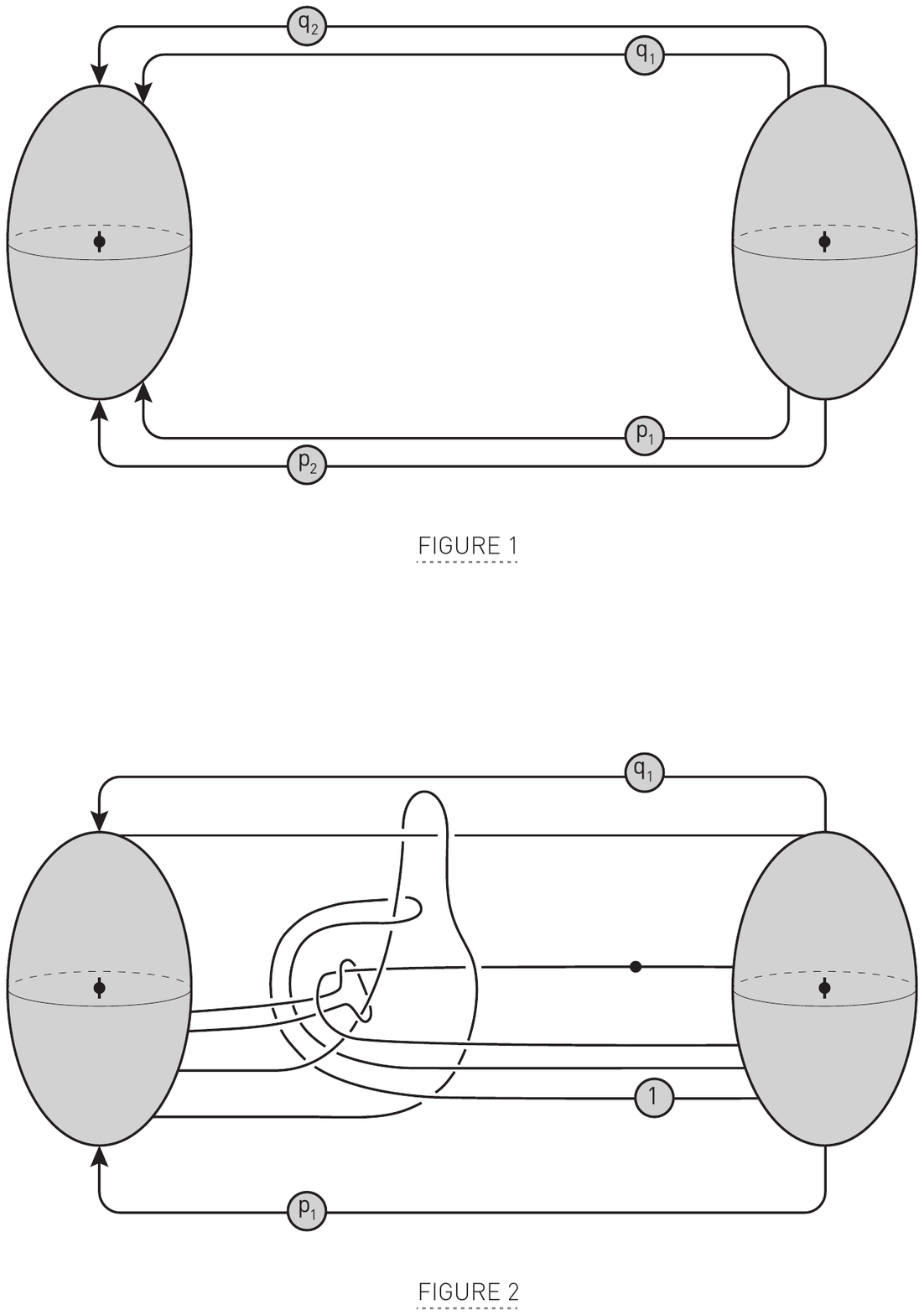}
\end{center}
\end{figure}

\subsection{Gluck twists and the proof of Theorem \ref{Theorem Pr}}{\label{Section Gluck}} The claim regarding Gluck twists of Theorem \ref{Theorem Pr} is a corollary of the study done in \cite{[AT]} of the handlebody structure of $A$. It was shown in that paper that applying a Gluck twist to an embedded 2-sphere in $S^3\widetilde{\times} S^1 \# S^2\times S^2$ produces the exotic manifold $A$ that was constructed in \cite{[AO]}. This is proven by constructing a handlebody of $A$ as in Figure 2 on page 15 with the $p_1$- and $q_1$- framed 2-handles removed \cite[Figure 7]{[AT]}: the dotted slice knot and the 2-handle with -1 framing expose the Gluck twist. We follow the notation in \cite{[AO]} for non-orientable 1-handles, and that of \cite{[A0]} for the orientable 1-handles.\\

\begin{figure}{\label{Figure 2}}
\begin{center}
\includegraphics[viewport= 430 150 150 730]{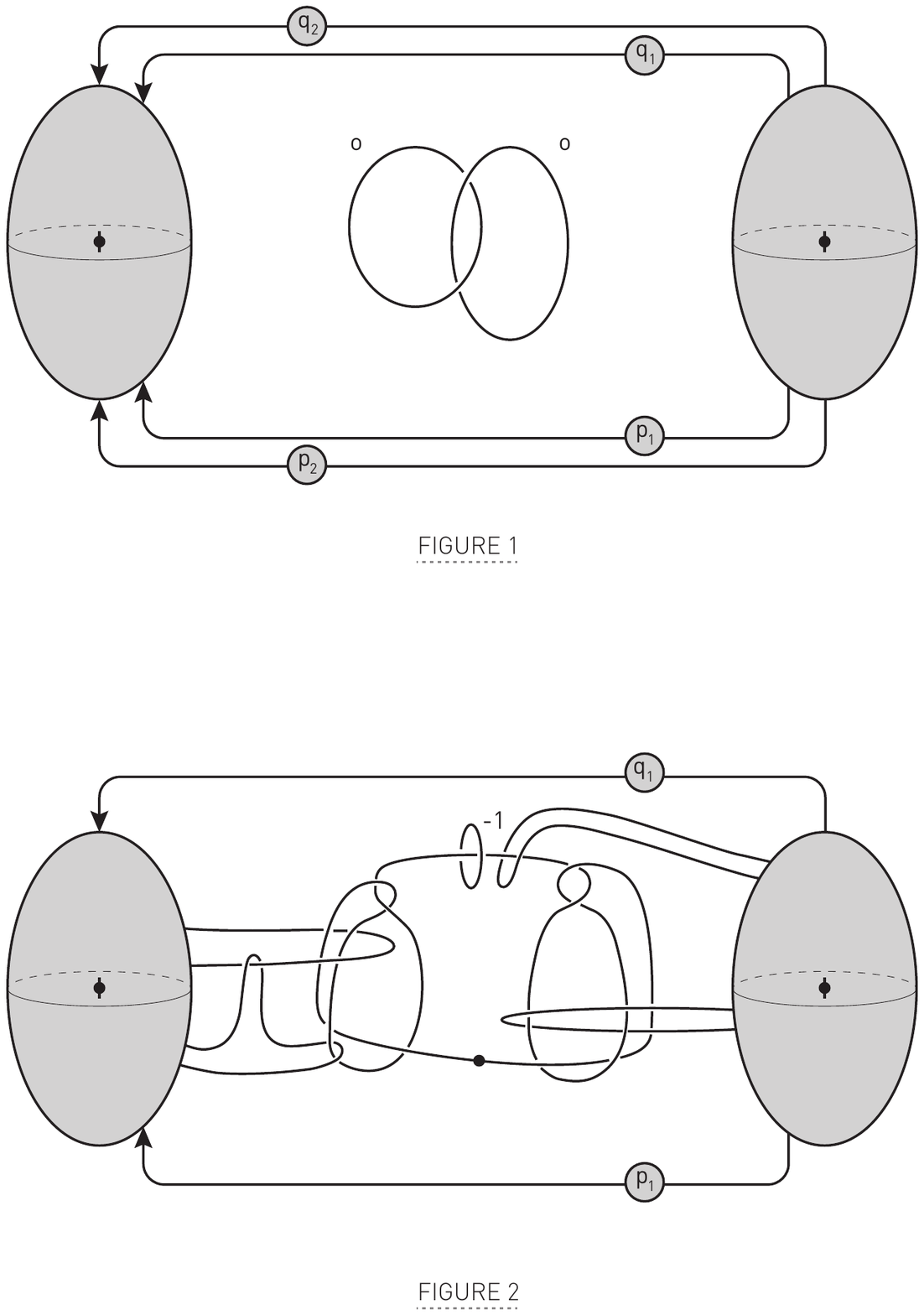}
\end{center}
\end{figure}

There is a handlebody of a circle sum of Definition \ref{Definition CS} of the manifold $A$ with a non-orientable 4-manifold $X$, which contains the very same slice knotted 1-handle and the -1 framed 2-handle by construction. Indeed, let  \begin{center}$Y:= (X - D^3 \widetilde{\times} S^1) \cup A_0$,\end{center} where $A_0: = A - D^3\widetilde{\times} S^1$ is an exotic $D^3\widetilde{\times} S^1 \# S^2\times S^2$. We glue along the common boundary $S^2\widetilde{\times} S^1$, the total space of the twisted $S^2$-bundle over $S^1$, and we take the identity map as gluing diffeomorphism. A handlebody of $Y$ is obtained by turning the handlebody of $X - D^3 \widetilde{\times} S^1$ upside down (cf. \cite[Example 5.5.5]{[GS]}), and joining it with $\partial A_0$ of Figure 2 on page 15 with the $p_1$- and $q_1$- framed 2-handles removed (cf. \cite[Example 5.5.8]{[GS]}).\\

Figure 2 on page 15 exemplifies the procedure on $S(2\gamma \oplus \R) \# S^2\times S^2$, where $\gamma$ denotes the real Hopf bundle over the real projective plane $\R P^2$. The total space of this non-orientable $S^2$-bundle over $\R P^2$ can be constructed as the double of the twisted disk bundle over the real projective plane, which is denoted by $D^2\widetilde{\times} \R P^2$. A handlebody of $D^2\widetilde{\times} \R P^2$ is given as Figure 1 with the linked  0-framed 2-handles and the $p_2$- $q_2$-framed 2-handles removed \cite[Section 0]{[AO]}. A handlebody of $S(2\gamma \oplus \R) \# S^2\times S^2$ is given in Figure 1, where $p_i + q_i$ is assumed to be odd.  Turning the handlebody of $D^2\widetilde{\times} \R P^2$ upside down, yields a handlebody that consists of one 2-handle, one 3-handle, and one 4-handle. Attaching these handles to $\partial A_0$ yields an exotic manifold obtained by a Gluck twist as illustrated in Figure 2 on page 15.\\

Using $-id$ as attaching map in the previous construction yields the examples of exotic smooth structures on $\#_{S^1} r \cdot \R P^4 \# S^2\times S^2$.

\section{Smooth structures on the orientation 2-covers and orientation-reversing $G$-involutions}{\label{Section Inv}}

We now show that the smooth structure of the orientation 2-covers of the manifolds constructed is standard, yielding examples of exotic $G$-involutions in the sense of Definition \ref{Definition Inv} for $G = \Z/2, \Z$, $\Z/2\ast \Z/2$, $\Z\rtimes \Z$, and $\Z\times \Z\times \Z\rtimes \Z$.

\subsection{Inequivalent $\Z/2$-involutions on connected sums of $S^2\times S^2$} Examples of such fixed point free involutions on connected sums of an even number of copies of $S^2\times S^2$ have been previously constructed in \cite{[FS], [CS], [G0], [A]}. The following proposition proves Item (1) of Theorem \ref{Theorem Inv} given the existence of a diffeomorphism \begin{equation}Q\#_{S^1} \R P^4 \rightarrow (\R P^4 \#_{S^1} \R P^4) \#_{S^1} M_A.\end{equation}

\begin{proposition}{\label{Proposition S}} The universal cover $\pi: \widetilde{M}_{n - 1}\overset{\pi}{\longrightarrow} Q\#_{S^1}(\#_{S^1}(n - 1)\cdot \R P^4)$ is diffeomorphic to the connected sum $(n - 1)(S^2\times S^2)\# S^4$ for $n\in \N$. The same conclusion holds for circle sums of copies of $Q$, and for hybrid circle sums of arbitrary number of copies of $Q$ and $\R P^4$.
\end{proposition}

\begin{proof} It is proven in \cite{[G0]} that $\widetilde{M_0}$ is diffeomorphic to $S^4$. Suppose $n = 2$, and consider $\pi: \widetilde{M}_1\longrightarrow Q \#_{S^1} \R P^4$. Abuse notation and let $\rho$ be the nontrivial loop generating $\pi_1(Q)$ and $\pi_1(\R P^4)$, and denote by $\nu(\rho)$ its tubular neighborhood. The circle sum can then be written as\begin{equation}
(Q - \nu(\rho)) \cup (\R P^4 -\nu(\rho)).
\end{equation}

The symbol "$\approx$" denotes the existence of a diffeomorphism. Using the hypothesis that the smooth structure on the universal cover of $Q$ is the standard one, we have \begin{equation}{\label{Decomp 1}} \widetilde{M}_1 \approx (S^4 - \pi^{-1}(\nu(\rho))) \cup (S^4 - \pi^{-1}(\nu(\rho))),\end{equation} since the corresponding loops inside $S^4$ are isotopic. Notice that the same decomposition holds true if we take only copies of $Q$ in the circle sum. We are abusing notation, denoting any covering map by $\pi$. Decomposing the 4-sphere as\begin{equation}{\label{Decomp 2}}
S^4 \approx \partial D^5 \approx \partial (D^3\times D^2) \approx (D^3\times S^1) \cup (S^2\times D^2),
\end{equation}

and substituting in \ref{Decomp 1} we obtain \begin{equation}{\label{Decomp 3}} \widetilde{M}_1 \approx (S^2\times D^2) \cup (S^2\times D^2) \approx S^2\times S^2,\end{equation} as it was claimed.

Assume $n = 3$, and consider \begin{equation}{\label{Decomp 4}}
\widetilde{M}_2 \overset{\pi}{\longrightarrow} (Q\#_{S^1}\R P^4)\#_{S^1} \R P^4.
\end{equation}

Decompose the circle sum associatively as\begin{equation}{\label{Decomp 5}}
((Q\#_{S^1}\R P^4) - \nu(\rho)) \cup (\R P^4 - \nu(\rho)).
\end{equation}

Employing the hypothesis on the smooth structure on the universal cover of $Q$ and \ref{Decomp 3}, we have the decomposition\begin{equation}
\widetilde{M}_2 \approx (S^2\times S^2 - \pi^{-1}(\nu(\rho))) \cup (S^4 - \pi^{-1}(\nu(\rho))).
\end{equation}

Again, notice that the decomposition holds true regardless of the number of copies of $Q$ involved in the circle sum of \ref{Decomp 4}. Write\begin{equation}{\label{Decomp 5}}
(S^2\times S^2 - \pi^{-1}(\nu(\rho))) \cup (S^2\times D^2)
\end{equation}

as\begin{equation}{\label{Decomp 6}}
((S^1\times (S^3 - D^3 \sqcup D^3))\cup(D^2\times S^2 \sqcup S^2\times D^2) - \pi^{-1}(D(\rho))) \cup (S^2\times D^2) \approx \end{equation} \begin{equation} \approx (S^1\times (S^3 - D^3 \sqcup D^3 \sqcup D^3)) \cup (D^2\times (S^2\sqcup S^2 \sqcup S^2)) \approx \end{equation} \begin{equation} \approx (S^2\times S^2) \# (S^2\times S^2).
\end{equation}

The remaining cases follow from an iteration of the previous arguments to the decomposition\begin{equation}{\label{Decomp 6}}
m(S^2\times S^2) \approx (S^1\times (S^3 - \underbrace{D^3 \sqcup \cdots \sqcup D^3}_{m + 1})) \cup (D^2\times (\underbrace{S^2 \sqcup \cdots \sqcup S^2}_{m + 1})).
\end{equation}
\end{proof}

\subsection{Inequivalent involutions by infinite groups} We now complete the proof of Theorem \ref{Theorem Inv} in the following proposition.

\begin{proposition}{\label{Proposition Inv}} Let $k\in \N$. The connected sums \smallskip
\begin{enumerate}
\item $S^3\times S^1 \# 2k(S^2\times S^2)$,\smallskip\smallskip

\item $T^2\times S^2 \# 2k(S^2\times S^2)$, \smallskip\smallskip

\item $T^2\times T^2 \# 2k(S^2\times S^2)$, and \smallskip\smallskip

\item $S^3\times S^1 \# 2(k - 1)(S^2\times S^2)$\smallskip
\end{enumerate}

admit free orientation-reversing exotic $G$-involutions by the groups $\Z$, $\Z \rtimes \Z$, $\Z\times \Z \times \Z \rtimes \Z$, and $\Z/2\ast \Z/2$, respectively.

\end{proposition}

As it was mentioned in the Introduction, Item (1) has been previously proven in \cite{[FS1]}. Item (2) yields examples of exotic involutions on two homeomorphism classes of orbit spaces $\xi_3 \# k(S^2\times S^2)$ and $\xi_5 \# k(S^2\times S^2)$.
 
\begin{proof} We begin with the proof of Item (1). Let $A_k$ be the manifold that is homeomorphic but not diffeomorphic to $S^3\widetilde{\times} S^1 \# k(S^2\times S^2)$ of Theorem \ref{Theorem M} built in \cite{[AO]} for $k\in \N$. We claim that the smooth structure of its oriented 2-cover $\widetilde{A}_k$ is standard, i.e., it is diffeomorphic to $S^3\times S^1 \# 2k(S^2\times S^2)$. Consider the case $k = 1$. We have $A_1 = (Q\# S^2\times S^2 - \nu(\R P^2)) \cup (D^3 \widetilde{\times} S^1)$ \cite{[AO]}. Since the universal cover $\widetilde{Q}$ is diffeomorphic to $S^4$ \cite{[G0]} (see Convention \ref{Convention 1}), the oriented two-cover is \begin{equation}\widetilde{A}_1 = (S^2\times D^2\cup D^3\times S^1 \# 2(S^2\times S^2) - S^2\times D^2) \cup (D^3\times S^1).\end{equation}
The 4-sphere is being considered as the boundary $\partial D^5 = \partial (D^3\times D^2)$. Therefore, \begin{equation}\widetilde{A}_1 = D^3\times S^1 \# 2(S^2\times S^2) \cup D^3\times S^1 \cong S^3\times S^1\# 2(S^2\times S^2).\end{equation} The claim now follows by induction on $k$ by taking connected sums with copies of $S^2\times S^2$. Theorem \ref{Theorem M} implies that the orbit spaces are homeomorphic, but not diffeomorphic, yielding exotic $\Z$-involutions on the connected sum $S^3\times S^1 \# 2k (S^2\times S^2)$ for every $k\in \N$ (cf. \cite{[FS1]}).

Items (2) - (3) follow from Item (1), since the cut-and-paste construction on the orientation 2-covers associated to the circle sum has the standard smooths structure by Item (1).

Regarding Item (4), we have the following. The kernel of the diagonal group homomorphism $\Z/2\ast \Z/2 \rightarrow \Z/2$ is $\Z$. Therefore, there is an oriented two-cover for the manifolds constructed in Section \ref{Section Exotica} that has infinite cyclic fundamental group. The techniques used to prove Item (1) show that this cover has the standard smooth structure. Thus, 
$S^3\times S^1 \# 2(k - 1)(S^2\times S^2)$ admits an exotic $\Z/2 \ast \Z/2$-involution for every $k\in \N$.

\end{proof}

\begin{remark}{\label{Remark Orb}}\begin{itemize}\item Item (5) in Theorem \ref{Theorem Inv}  can be compared to \cite[Theorem 1.3]{[JK]}. A topological involution on $S^1\times S^3$ whose quotient is $\R P^4\# \R P^4$ is topologically conjugate to the involution $(t, x)\mapsto (\overline{t}, - x)$, where $\overline{t}$ is the complex conjugate of $t\in S^1$.

\item Examples of inequivalent $G \ast \Z/p_1 \ast \cdots \ast \Z/p_r$ - involutions: the $\eta$-invariant can be employed to distinguish the orbit spaces of involutions by the free product of groups $G \ast \Z/p_1 \ast \cdots \ast \Z/p_r$ using connected sums of manifolds constructed above and $\Q$-homology 4-spheres with fundamental group of odd order and vanishing $\omega_2$.



\end{itemize}
\end{remark}

\end{document}